\documentclass[12pt]{amsart}

\usepackage{breakurl}
\usepackage{graphicx}
\usepackage{amsmath}
\usepackage{amscd}
\usepackage{amsfonts}
\usepackage{amssymb}
\usepackage{color}
\usepackage{fullpage}
\usepackage{mathtools}
\usepackage[pagebackref,hypertexnames=false, colorlinks, citecolor=red,linkcolor=blue, urlcolor=red]{hyperref}
\usepackage{todonotes}
\usepackage{enumitem}

\numberwithin{equation}{section}

\newtheorem{theorem}{Theorem}[section]
\newtheorem{lemma}[theorem]{Lemma}
\newtheorem{proposition}[theorem]{Proposition}

\newtheorem{corollary}[theorem]{Corollary}

\theoremstyle{definition}

\newtheorem{remark}[theorem]{Remark}

\newcommand{\be}{\begin{equation}}
\newcommand{\ee}{\end{equation}}
\newcommand{\bes}{\begin{equation*}}
\newcommand{\ees}{\end{equation*}}

\newcommand{\cA}{\mathcal{A}}
\newcommand{\cB}{\mathcal{B}}

\newcommand{\cD}{\mathcal{D}}

\newcommand{\cF}{\mathcal{F}}

\newcommand{\cH}{\mathcal{H}}

\newcommand{\cL}{\mathcal{L}}

\newcommand{\cS}{\mathcal{S}}
\newcommand{\cT}{\mathcal{T}}
\newcommand{\cU}{\mathcal{U}}

\newcommand{\cW}{\mathcal{W}}
\newcommand{\cX}{\mathcal{X}}

\newcommand{\bB}{\mathbb{B}}
\newcommand{\bC}{\mathbb{C}}

\newcommand{\bE}{\mathbb{E}}

\newcommand{\bM}{\mathbb{M}}
\newcommand{\bN}{\mathbb{N}}

\newcommand{\bR}{\mathbb{R}}

\newcommand{\bZ}{\mathbb{Z}}

\newcommand{\ol}{\overline}

\newcommand{\re}{\operatorname{Re}}
\newcommand{\im}{\operatorname{Im}}

\newcommand{\sa}{{\mathrm{sa}}}

\newcommand{\alg}{\operatorname{alg}}

\newcommand{\UCP}{\operatorname{UCP}}

\newcommand{\Wmin}[1]{\cW^{\mathrm{min}}_{#1}}
\newcommand{\Wmax}[1]{\cW^{\mathrm{max}}_{#1}}

\newcommand{\fB}{{\mathfrak{B}}}

\newcommand{\fD}{{\mathfrak{D}}}

\newcommand{\FORAL}{\text{ for all }}

\begin{document}

\title{On the matrix range of random matrices}

\author{Malte Gerhold}
 \address{M.G., Institut f\"ur Mathematik und Informatik \\
 Universit\"at Greifswald\\Walther-Rathenau-Stra\ss{}e 47 \\
 17487 Greifswald \\ Germany\\}
 \email{mgerhold@uni-greifswald.de}
\urladdr{\href{https://math-inf.uni-greifswald.de/institut/ueber-uns/mitarbeitende/gerhold/}{\url{https://math-inf.uni-greifswald.de/institut/ueber-uns/mitarbeitende/gerhold/}}}

 \author{Orr Moshe Shalit}
 \address{O.S., Faculty of Mathematics\\
 Technion - Israel Institute of Technology\\
 Haifa\; 3200003\\
 Israel}
 \email{oshalit@technion.ac.il}
 \urladdr{\href{https://oshalit.net.technion.ac.il/}{\url{https://oshalit.net.technion.ac.il/}}}

 \thanks{The work of M. Gerhold is partially supported by the DFG, project no.\  397960675.}
 \thanks{The work of O.M. Shalit is partially supported by ISF Grant no.\ 195/16.
 }
 
\addcontentsline{toc}{section}{Abstract}

\begin{abstract} 
This note treats a simple minded question: {\em what does a typical random matrix range look like?}
We study the relationship between various modes of convergence for tuples of operators on the one hand, and continuity of matrix ranges with respect to the Hausdorff metric on the other.
In particular, we show that the matrix range of a tuple generating a continuous field of C*-algebras is continuous in the sense that every level is continuous in the Hausdorff metric.
Using this observation together with known results on strong convergence in distribution of matrix ensembles, we identify the limit matrix ranges to which the matrix ranges of independent Wigner or Haar ensembles converge.
\end{abstract}
\maketitle


\section{Introduction}\label{sec:introduction}

Let $M_n = M_n(\bC)$ denote the set of all $n \times n$ matrices over $\bC$, and let $M_n^d$ be the set of all $d$-tuples of such matrices.
The {\em matrix range} \cite{Arv72} of a tuple $A = (A_1, \ldots, A_d)$ in $B(\cH)^d$ is the disjoint union $\cW(A) = \bigcup_{n} \cW_n(A)$, where for all $n\in\mathbb N$, the set $\cW_n(A) \subseteq M_n^d$ is defined by
\[
\cW_n(A) = \bigl\{\bigl(\phi(A_1), \ldots, \phi(A_d)\bigr): \phi \in \UCP\bigl(B(\cH), M_n \bigr)\bigr\}; 
\]
here and below, UCP stands for {\em unital completely positive}, and $\UCP(B(\cH), M_n)$ is the set of all UCP maps from $B(\cH)$ to $M_n$.
Matrix ranges are interesting for several reasons.
First, a set $\cS = \bigcup_n \cS_n \subseteq \bM^d := \bigcup_{n=1}^\infty M_n^d$ is a bounded closed matrix convex set if and only if $\cS = \cW(A)$ for some $A\in B(\cH)^d$; 
also, the matrix range of $A$ is a complete invariant of the operator system generated by $A$, and it is useful when considering interpolation problems for UCP maps (see \cite{DDSS17}).
Moreover, in the case of a fully compressed tuple $A$ of compact operators or normal operators, the matrix range determines $A$ up to unitary equivalence \cite{PS19}.
Finally, the first level $\cW_1(A)$, which we loosely refer to as the {\em numerical range} of the tuple $A$, is closely related to the so-called {\em joint numerical range} of $A_1, \ldots, A_d$ (see, e.g., \cite{LP00,LP11,LP19}).
See also \cite{Eve18,Passer,PSS18} for other recent works where matrix ranges are used and studied.

The current interest in matrix ranges led us to ask: {\em what does a typical matrix range look like?}
The purpose of this paper is to formulate and then answer this question in precise terms.
Our main results are Theorems \ref{thm:wigner} and \ref{thm:haar}, which say that the matrix ranges of a sequence $T^N = (T^N_1, \ldots, T^N_d)$ of $d$ independent $N \times N$ matrices from certain ensembles converge almost surely to the matrix range of a corresponding tuple of operators $t$, in the sense that $\cW_n(T^N) \xrightarrow{N \to \infty} \cW_n(t)$ in the Hausdorff metric for compact subsets of $M_n^d$, for all $n$. If $T^N$ are drawn from a Wigner ensemble (satisfying certain moment conditions) then $t = s$ is a $d$-tuple of free semicirculars (Theorem \ref{thm:wigner}); if $T^N$ is drawn according to the Haar distribution on the unitary group $\cU_N$,  then $t = u$ is a tuple of $d$ free Haar unitaries (Theorem \ref{thm:haar}).
In particular, we recover (part of) the results of Collins, Gawron, Litvak and \.{Z}yczkowski \cite{CGLZ14} on the almost sure convergence of the numerical range of random matrices by looking at the case $d = 2$ and $n=1$ of the case of Wigner distributions (Corollary \ref{cor:randomNumRange}).

In Section \ref{sec:cont-matr-range} we prove an effective Effros-Winkler type separation theorem (Lemma \ref{lem:cont_matrange}), which is then used to show that for a family of $d$-tuples $\{\xi_t\}$, if $\lim_{t\to t_0} \|p(\xi_t)\| = \|p(\xi_{t_0})\|$ for every $*$-polynomial $p$, then $\cW_n(\xi_t)$ converges to $\cW_n(\xi_{t_0})$ for all $n$ (see Theorem \ref{thm:convergences}; a certain converse is also proved in Theorem \ref{thm:convergences_unitary}). 
Combining our results with Anderson's theorem \cite{And13} that $\|p(X^N)\| \xrightarrow{a.s.} \|p(s)\|$ for Wigner ensembles, and Collins and Male's result \cite{CM14} that $\|p(U^N)\| \xrightarrow{a.s.} \|p(u)\|$ for Haar ensembles (for every $*$-polynomial $p$), we obtain our main results, Theorems \ref{thm:wigner} and \ref{thm:haar}.

In Section \ref{sec:description}, we use Lehner's formulas \cite{Leh99} to describe the limiting matrix ranges $\cW(s)$ and $\cW(u)$, and in particular the first levels $\cW_1(s)$ and $\cW_1(u)$ for the free semicircular tuple $s$ and the free Haar tuple $u$ (Propositions \ref{prop:LehnerBalllevel1} and \ref{prop:LehnerPolydisclevel1}).
Since $\cW_1(s)$ is a Euclidean ball of radius two, $\cW(s/2)$ lies over a Euclidean ball of radius one, and it is reasonable to refer to it as a {\em matrix ball}.
We compare it to several other matrix balls that were considered in the literature, and we find that $\cW(s/2)$ is different from every one of the five immediate suspects $\Wmin{}(\ol{\bB}_d)$, $\cB$, $\fD$, $\cB^\bullet$ and $\Wmax{}(\ol{\bB}_d)$ (see the next subsection for definitions of these matrix balls).

\subsection*{Some definitions and notation}

In this paper, $d$ will always be some positive integer that may be considered as fixed throughout. If not indicated otherwise, sums will be assumed to run from $1$ to $d$ over all appearing indices.
We will use $\bB_d$ to denote the open unit ball in $\bR^d$.
We let $M_n = M_n(\bC)$ denote the set of all $n \times n$ matrices over $\bC$, and $M_n^d$ the set of all $d$-tuples of such matrices.
The ``noncommutative universe'' (in $d$ variables) is the disjoint union $\bM^d = \bigcup_{n=1}^\infty M_n^d$.
It is sometimes convenient to work in the ``noncommutative selfadjoint universe'' $\bM^{d}_{\sa} = \bigcup_{n=1}^\infty (M^d_n)_{\sa}$ instead of $\bM^d$. One can always move from $\bM^d$ to $\bM^{2d}_{\sa}$ by replacing $(A_1, \ldots, A_d)$ with the $2d$-tuple of real and imaginary parts $(\re A_1, \im A_1, \ldots, \re A_d, \im A_d)$, and for many problems in matrix convexity there is no loss of generality working in the selfadjoint setting.

We shall consider subsets $\cS$ of $\bM^d$ (or $\bM^d_{\sa}$), and write $\cS_n$ or $\cS(n)$ for the {\em $n$th level} $\cS \cap M_n^d$ (or $\cS \cap (M_n)^d_{\sa}$).
Such a set $\cS \subseteq \bM^d$ is said to be {\em matrix convex}, if it is closed under direct sums and under the application of UCP maps, that is, whenever $X$ is in $\cS_n$ and $\phi\colon M_n \to M_k$ is a UCP map, the tuple $\phi(X) := (\phi(X_1), \ldots, \phi(X_d))$ is in $\cS_k = \cS \cap M_k^d$.
Equivalently, $\cS$ is matrix convex if and only if it closed under {\em matrix convex combinations}.

We let $\bC\langle z \rangle = \bC\langle z_1,\ldots,z_d\rangle$ denote the complex algebra of polynomials in $d$ noncommuting variables $z_1, \ldots, z_d$, and we write $\bC\langle z, z^* \rangle$ for the $*$-algebra $\bC\langle z_1,\ldots,z_d, z_1^*, \ldots, z_d^*\rangle$ of $*$-polynomials in these variables. $\bC_k\langle z, z^*\rangle$ will denote the space of $*$-polynomials of degree less than or equal to $k$ (we will be interested really just in the case $k=1$). We write $M_n(\bC \langle z, z^*\rangle)$ for the space of $*$-polynomials in $z$ with $n \times n$ matrix coefficients, or --- equivalently --- $n \times n$ matrices with entries in $\bC \langle z, z^* \rangle$. Likewise, $M_n(\bC_k \langle z, z^*\rangle)$ stands for the subspace consisting of matrix valued $*$-polynomials of degree at most $k$, etc. 

For every $d$-tuple of operators $X = (X_1, \ldots, X_d)$ we write 
\[
\|X\| := \left\|\sum X_i X_i^*\right\|^{1/2} .
\]
For every fixed $n$, $M_n^d$ then becomes a metric space with $d(X,Y) = \left\|X-Y\right\|$.
The norm on $M_n^d$ also induces a distance function on the subsets of $M_n^d$, the {\em Hausdorff distance}, given by
\[
d_H(E,F) = \max\left\{\sup_{x \in E} d(x,F),  \sup_{y \in F}  d(y,E)\right\},
\]
where $d(x,F) = \inf_{y \in F}d(x,y)$. 
The {\em Hausdorff metric}, also denoted $d_H$, is the restriction of the Hausdorff distance to the compact subsets of $M_n^d$. 

We say that a matrix convex set $\cS$ is {\em closed}, if $\cS_n$ is closed for all $n$.
We say that it is {\em bounded}, if there exists some $R>0$ such that $\|X\| \leq R$ for all $X \in \cS$.
We say that $0 \in \operatorname{int}(\cS)$, if there exists some $r>0$ such that $\cS$ contains all $X \in \bM^d$ satisfying $\|X\|<r$.
It is not hard to see (see, for example \cite[Lemma 3.4]{DDSS17})
that $\cS$ is bounded if and only if $\cS_1$ is bounded, and that $0 \in \operatorname{int}(\cS)$ if and only if $0 \in \operatorname{int}(\cS_1)$ in the usual sense.

For every closed and convex set $K$ in $\bR^d$ (or in $\bC^d$) there exist minimal and maximal matrix convex sets, denoted $\Wmin{}(K)$ and $\Wmax{}(K)$, respectively, that have $K$ as its first level \cite[Section 4]{DDSS17}.
We will need below that $\Wmax{}(\ol{\bB}_d) \subseteq d \Wmin{}(\ol{\bB}_d)$; see \cite[Proposition 7.11]{DDSS17} and \cite[Proposition 14.1]{HKMS19} (in fact, the containment $\Wmax{}(K) \subseteq d \Wmin{}(K)$ is true for every convex $K\subseteq \bR^d$ satisfying $K = - K$, see \cite[Theorem 5.8]{FNT17} and \cite[Corollary 4.4]{PSS18}).

The {\em polar dual} \cite{EW97} of a matrix convex set $\cS \subseteq \bM^d$ is the set
\[
\cS^\circ = \left\{X \in \bM^d : \re \sum X_i \otimes Y_i \leq I \FORAL Y \in \cS\right\}.
\]
In the selfadjoint context $\cS \subseteq \bM_{\sa}^d$, we use the definition
\[
\cS^\bullet = \left\{X \in \bM^d_{\sa} : \sum X_i \otimes Y_i \leq I \FORAL Y \in \cS\right\}.
\]
Some properties of polar duals are recorded in \cite[Section 3]{DDSS17}.
We will need the identity
\[
\cW(A)^\circ = \cD_A := \left\{X \in \bM^d : \re \sum X_i \otimes A_i \leq I\right\}
\]
and, in the selfadjoint context,
\[
\cW(A)^\bullet = \cD_A^{\sa} := \left\{X \in \bM^d_{\sa} :  \sum X_i \otimes A_i \leq I\right\}.
\]
When $0 \in \cW(A)$, then also $\cW(A) = (\cD_A)^\circ$ (or $\cW(A) = (\cD_A^{\sa})^\bullet$).
The sets of the form $\cD_A$ (or $\cD^{\sa}_A$)  are called {\em free spectrahedra} (some authors reserve this term for matrix tuples $A$, but we do not).

Polar duality of (selfadjoint) matrix convex sets is consistent with polar duality of (real) convex sets for the first levels, i.e., if $\cS$ is a (selfadjoint) matrix convex set with first level $\cS_1=K$, then
\begin{align*}
(\cS^{\circ})_1=K'&:=\left\{w\in\mathbb C^d : \re\sum w_iz_i\leq 1 \textrm{ for all } z \in K\right\}
\end{align*}
(or $(\cS^{\bullet})_1=K':= \left\{w\in\mathbb R^d : \sum w_i z_i\leq 1 \textrm{ for all } z \in K\right\}$, where we use the same notation for polar duality in $\mathbb R^d$). 

The most natural selfadjoint matrix ball to consider is perhaps
\[
\fB = \left\{X \in \bM^{d}_{\sa} : \sum X_i^2 \leq 1\right\}.
\]
Sometimes we write $\fB_d$ if we need to emphasize the dependence on the ``number of variables'' $d$. The polar dual $\fB^\bullet$ also seems worthy of consideration.
Another interesting matrix ball is
\[
\fD = \left\{X \in \bM^{d}_{\sa} : \sum X_i \otimes \ol{X}_i \leq 1\right\}.
\]
$\fD$ is self-dual in the sense that $\fD^\bullet = \fD$ \cite[Lemma 9.2]{DDSS17} (and in fact $\fD$ is the only self-dual matrix convex set in $\bM^d_{\sa}$ that is closed under complex conjugation \cite[Remark 9.3]{DDSS17}). 
The minimal matrix convex set over the unit ball $\Wmin{}(\ol{\bB}_d)$, and the maximal matrix convex set over the unit ball $\Wmax{}(\ol{\bB}_d)$, are two other compelling matrix balls to consider.
In \cite[Corollary 9.4]{DDSS17}, it was shown that
\be\label{eq:balls1}
\Wmin{}(\ol{\bB}_d) \subset \fB \subset \fD = \fD^\bullet \subset \fB^\bullet \subset \Wmax{}(\ol{\bB}_d),
\ee
and that all inclusions are strict.
The question of the scales that make the reverse inclusions hold was studied in \cite[Section 9]{DDSS17} and \cite[Section 14]{HKMS19}.
It was shown that
\be\label{eq:balls2}
\frac{1}{\sqrt{d}}\Wmax{}(\ol{\bB}_d) \subset \fB \subset \fD \subset \fB^\bullet \subset \sqrt{d}\Wmin{}(\ol{\bB}_d).
\ee
Note that a matrix convex set $\cS \subseteq \bM^d_{\sa}$ is bounded if and only if $\cS \subseteq R \fB$, and that $0 \in \operatorname{int}(\cS)$ if and only if $r\fB \subseteq \cS$ (for the same $R$ and $r$ as in the definition above). By \eqref{eq:balls1} and \eqref{eq:balls2}, we can replace $\fB_d$ with any one of the sets $\Wmin{}(\ol{\bB}_d)$, $\fD$, $\fB^\bullet$ or $\Wmax{}(\ol{\bB}_d)$ in the definition of boundedness or $0$ in the interior (though the inclusions will hold with different values of $r$ and $R$).

\section{Continuity of matrix ranges}\label{sec:cont-matr-range}

We begin with an effective version of the Effros-Winkler Hahn-Banach type separation theorem \cite{EW97}.
For conciseness, we state it in the selfadjoint setting, but it should be clear how to obtain a nonselfadjoint version.
For every $n$, we define a norm on the finite dimensional space $M_n(\bC_1 \langle z\rangle)$ of linear pencils of size $n$, by
\[
\|p\| = \sup\{\|p(X)\| : X \in \bM^d_{\sa}, \|X\| \leq 1\}.
\]

\begin{lemma}\label{lem:cont_matrange}
Let $\cS \subseteq \bM^d_{\sa}$ be a bounded matrix convex set.
Then for every $\varepsilon > 0$, there is a $\delta>0$ such that for all $n \in \bN$, and all $A \in (M_n)^d_{\sa}$,
\[
d(A,\cS) > \varepsilon \implies \exists p \in M_n(\bC_1\langle z\rangle),\|p\|\leq 1 . \forall X\in \cS.  \left\|p(A)\right\| > \left\|p(X)\right\| + \delta .
\]
Furthermore, $\delta$ can be chosen to depend only on $\varepsilon$, $d$, and a positive number $R$ such that $\cS \subseteq R\fB_d$ (in particular, $\delta$ is independent of the level $n$ and the matrix convex set $\cS$).
\end{lemma}

\begin{proof}
We first prove the result under the assumption that
\[
r \fB_d \subseteq \cS \subseteq R\fB_d
\]
for some $0<r<R$.
Fix $\varepsilon > 0$, and assume without loss of generality that $\varepsilon < 1 < R$. Let $A\in (M_n^d)_{\sa}$ be such that $d(A,\cS) := d(A,\cS_n)>\varepsilon$. Then $A\notin c\cS$ for $c=1+\frac{\varepsilon}{R}$.
Indeed, $A\in c\cS$ would imply
\[
d(A,\cS) \leq \sup_{B\in \cS_n} d(cB,B) = \sup_{B\in \cS_n}\frac{\varepsilon}{R}\|B\|\leq \varepsilon.
\]
Note that $1+c>2$ and, by our assumption $\varepsilon < R$, also $1-\frac{1}{c}>\frac{\varepsilon}{2R}$.

By the Effros-Winkler separation theorem \cite[Theorem 5.4]{EW97}, we can find a monic linear pencil $p(z) = I+\sum a_i z_i$ of size $n$ with selfadjoint coefficients that separates $A$ from $c \cS$, that is, such that
\[
p(cX) = I + \sum a_i\otimes cX_i \geq 0 \quad \textrm{for all }  X\in \cS,
\]
while
\[
p(A) = I + \sum a_i\otimes A_i \not\geq 0.
\]
Since $r\fB_d$ is symmetric and a subset of $\cS$, we have $-I\leq \sum a_i\otimes B_i\leq I$
for all $B\in r\fB_d$ and, therefore,
\[
\sup_{B\in r \fB_d}\left\|\sum a_i\otimes B_i\right\|\leq 1.
\]
So the linear pencil $p$ has norm $\|p\| \leq 1 + r^{-1}$.

From $p(A)\not\geq0$, we know that there is a negative eigenvalue of $p(A)$. In the following, we establish estimates for the minimal and maximal eigenvalues of $p(X)$ for $X \in \cS$. First, using $p(cX) \geq 0$, we know that
\[
\lambda_{\min}\left(I+\sum a_i\otimes X_i\right)=\lambda_{\min}\left(\frac{1}{c}\left(I+\sum a_i\otimes cX_i\right)+ 1-\frac{1}{c}\right) \geq 1-\frac{1}{c}>\frac{\varepsilon}{2R}.
\]
On the other hand, $\cS \subseteq R\fB_d$, so if $X \in \cS$ we can write $X = R B$ with $B\in \fB_d$, thus
\[
\lambda_{\max}\left(I+\sum a_i\otimes X_i\right) = \left\|I+\sum a_i\otimes RB_i\right\|\leq R\|p\| + \left|1-R\right| \leq 2R(1 + r^{-1}).
\]
Now, we define a shifted linear pencil
\[
q := p - 2R(1 + r^{-1}).
\]
Then,
\[
\|q\|\leq (1+r^{-1}) + 2R(1 + r^{-1}) = (2R + 1)(1 + r^{-1}).
\]
Furthermore,
\[
\sigma(q(X)) \subseteq \left[\frac{\varepsilon}{2R}-2R(1+r^{-1}),0 \right]
\]
for all $X\in \cS_n$ and $\lambda_{\min}(q(A))< -2R(1 + r^{-1})$. Therefore $\|q(A)\|-\|q(X)\|>\frac{\varepsilon}{2R}$ for all $X \in \cS$.

Now, $\widetilde{q}:=\frac{1}{(2R + 1)(1 + r^{-1})}q$ is a polynomial of norm less than 1 with
\[
\|\widetilde{q}(A)\|-\|\widetilde{q}(X)\|>\frac{\varepsilon}{2R(2R + 1)(1 + r^{-1})}=:\delta
\]
for all $X \in \cS$, as required.
Thus, the lemma is proved under the assumption that $\cS$ contains $0$ in its interior, with $\delta$ depending also on the size of a ball that fits in $\cS$.

Let $\cS$ be a general bounded (i.e., $\cS \subseteq R\fB_d$ for some $R$) matrix convex set, not necessarily containing $0$ in its interior.
By applying an affine change of variables we may assume that that $0 \in \cS$ (see, e.g., \cite[Section 3.1]{PSS18}).
Given $r > 0$, we define
\[
\cS^{(r)} = \cS + r \fB_d = \bigcup_{n=1}^\infty \left\{X + Y : X \in \cS(n), Y \in r\fB_d(n)  \right\},
\]
which is readily seen to be a bounded matrix convex set containing $r \fB_d$.
Now if $d(A,\cS) > \varepsilon$, then $d(A, \cS^{(\varepsilon/2)}) > \varepsilon/2$, and we may proceed as in the first part of the proof.
\end{proof}

In the proof of Theorem \ref{thm:convergences} below, we will need that if $\cW_n(A)$ and $\cW_n(B)$ both contain a neighborhood of zero and are close in the Hausdorff metric, then there exists a constant $c$ not much bigger than $1$ such that $\cW_n(A) \subseteq c \cW_n(B)$.
This follows from the following simple lemma, which we state more generally.

\begin{lemma}\label{lem:close_implies_dilation}
  Let $(\cX, \|\cdot\|)$ be a normed space
  and $\varepsilon>0$. 
Suppose that $E,F \subseteq \cX$ are bounded convex sets, and assume that there exists some $r > 0$ such that the ball $B_r = \{ x \in \cX : \|x\| \leq 1\}$ is contained in both $E$ and $F$.
Then $d_H(E,F) < \varepsilon$ implies $F \subseteq \frac{r+\varepsilon}{r}E$.
\end{lemma}
\begin{proof}
Let $y \in F$.
By the definition of the Hausdorff distance, there exists $x \in E$ such that $\|x - y\| < \varepsilon$.
Since $B_r \subseteq E$, $\frac{r}{\varepsilon}(y-x) \in E$.
Since $E$ is convex, the convex combination
\[
\frac{\varepsilon}{r+\varepsilon} \frac{r}{\varepsilon}(y-x) + \frac{r}{r+\varepsilon} x = \frac{r}{r+\varepsilon}y
\]
is also in $E$.
In other words, $y \in \frac{r+\varepsilon}{r}E$.
\end{proof}

In the proof of the next theorem, we will use the language of continuous bundles of C*-algebras according to \cite[Definition 1.1]{KW95}. 
In fact, in this paper, our interest lies only in bundles over the one point compactification of the natural numbers $\hat{\bN} = \bN \cup \{\infty\}$, in which case the notion of a {\em continuous bundle} coincides with the classical notion of a {\em continuous field} of C*-algebras in the sense of \cite[Section 10]{DixmierBook}. We formulate the following theorem in slightly greater generality since we have some future applications in mind.

\begin{theorem}\label{thm:convergences}
Suppose that $\cT$ is a metric space, and suppose that for every $t\in \cT$, there is a unital C*-algebra $\cA_t$ and a $d$-tuple $\xi_t = (\xi_{t,1}, \ldots, \xi_{t,d}) \in \cA_t^d$ generating $\cA_t$ as a unital C*-algebra. Consider the following conditions for $t_0 \in \cT$:
\begin{enumerate}[label={\textup{(\roman*)}}]
\item\label{it:polys} $\lim_{t\to t_0} \|p(\xi_t)\| = \|p(\xi_{t_0})\|$ for all $p\in \bC\langle z,z^* \rangle$.
\item\label{it:matpolys} $\lim_{t\to t_0} \|p(\xi_t)\| = \|p(\xi_{t_0})\|$ for all $n \in \bN$ and all $p\in M_n(\bC\langle z,z^* \rangle)$.
\item\label{it:linmatpolys} $\lim_{t\to t_0} \|p(\xi_t)\| = \|p(\xi_{t_0})\|$  for all $n\in \bN$ and all $p\in M_n(\bC_1 \langle z, z^*\rangle)$.
\item\label{it:matrange} $\lim_{t\to t_0}d_H(\cW_n(\xi_t),\cW_n(\xi_{t_0})) = 0$ for all $n\in\mathbb N$.
\end{enumerate}
Then we have the implications \ref{it:polys} $\Longleftrightarrow$ \ref{it:matpolys} $\Longrightarrow$ \ref{it:linmatpolys} $\Longleftrightarrow$ \ref{it:matrange}.
\end{theorem}
\begin{proof}
Since continuity of maps between metric spaces is determined by sequences, it suffices to consider the compact metric space $\cT = \hat{\bN} = \bN \cup \{\infty\}$ (the one point compactification of the natural numbers) and to test continuity at $t_0 = \infty$.

Consider the minimal bundle of unital C*-algebras over $\hat{\bN}$ with $\xi_1,\ldots, \xi_d$ as sections, i.e., the bundle $(\hat{\bN}, \cA_t,\cA)$ with base space $\hat{\bN}$, fiber $\cA_t$ at $t\in\hat\bN$, and bundle C*-algebra $\cA\subset \prod_{t\in \hat{\bN}} \cA_t$ generated by the sections $f1:=(t\mapsto f(t)1_{\cA_t})$ and  $f \xi_j := (t \mapsto f(t)\xi_{t,j})$, where  $j=1, \ldots, d$ and $f \in C(\hat{\bN})$ (with the canonical projections).  Assume \ref{it:polys}, which is easily seen to be equivalent to continuity of the bundle $(\hat{\bN}, \cA,\cA_t)$.  
Since $M_n$ is nuclear, \cite[Remark 2.6(1)]{KW95} implies that the bundle $(\hat{\bN}, \cA_t\otimes M_n, \cA \otimes M_n)$ is continuous, hence continuous at $\infty$, proving \ref{it:polys} $\Rightarrow$ \ref{it:matpolys}.
The converse implication, as well as the implication \ref{it:matpolys} $\Rightarrow$ \ref{it:linmatpolys}, are immediate.

Now suppose that \ref{it:linmatpolys} holds.
Note that \ref{it:linmatpolys} implies in particular that 
\[
\|\xi_k\| \xrightarrow{k \to \infty} \|\xi_\infty\|
\] 
(when considered as row norms), thus the family $\xi_k$ and,  therefore, the matrix ranges $\cW(\xi_k)$ remain uniformly bounded as $k \to \infty$.

For the remainder of the proof we shall argue in the selfadjoint setting, that is, we assume that all $\xi_k$ are selfadjoint and therefore all our matrix ranges are contained in $\bM^d_{\sa}$.
As in Lemma \ref{lem:cont_matrange}, we define a norm
\[
\| p\| = \sup_{\|X\| \leq 1}\|p(X)\| = \sup_{X \in \fB_d} \|p(X)\|
\]
on the finite dimensional spaces $M_n(\bC_1 \langle z\rangle)$.

Assume for contradiction that \ref{it:linmatpolys} holds but \ref{it:matrange} does not hold at some level $n_0$.
Then there is some $\varepsilon > 0$ such that $d_H(\cW_{n_0}(\xi_k), \cW_{n_0}(\xi_\infty)) > \varepsilon$ for infinitely many $k$, and by passing to a subsequence we may assume that this inequality holds for all $k$.
Therefore, for every $k$, at least one of the following two options happens:
\begin{enumerate}[label={\textup{(\arabic*)}}] 
\item\label{option:1} $\max_{A \in \cW_{n_0}(\xi_k)} d(A, \cW_{n_0}(\xi_\infty)) > \varepsilon$, or
\item\label{option:2} $\max_{A \in \cW_{n_0}(\xi_\infty)} d(A, \cW_{n_0}(\xi_k)) > \varepsilon$.
\end{enumerate}
Let us assume that the first option happens infinitely many times.
By passing to a subsequence yet again, we assume it holds for all $k$, so we actually have that for all $k$, there is an $A^{(k)} \in \cW_{n_0}(\xi_k)$ such that $d(A^{(k)}, \cW_{n_0}(\xi_\infty)) > \varepsilon$.

By Lemma \ref{lem:cont_matrange}, there
exist a fixed $\delta > 0$ and polynomials $p_{k}$ in the (compact) unit ball of $M_{n_0}(\bC_1 \langle z \rangle)$ such that for all $k$ and all $B \in \cW(\xi_\infty)$,
\[
\left \| p_{k}(\xi_k) \right\| \geq \left\|p_{k}\left(A^{(k)}\right)\right\|  > \left\|p_{k}(B)\right\| + \delta .
\]
Letting $p$ be a limit point of the sequence $(p_{k})_k$, and using the joint boundedness of the $\xi_k$, we have
\[
\left \| p(\xi_k) \right\| \geq \left \| p_{k}(\xi_k) \right \| - \left \| p_{k}(\xi_k)  - p(\xi_k) \right \| > \left\|p(B)\right\| + \delta/2,
\]
for all $k$ in an infinite subset $K \subseteq \bN$. From $\|p(\xi_\infty)\|=\sup_{B\in \cW(\xi_\infty)}\|p(B)\|$, we conclude
$\left \| p(\xi_k) \right\| \geq \left\|p(\xi_\infty)\right\| + \delta/2$
for all $k\in K$, which is a contradiction to \ref{it:linmatpolys}.

If option \ref{option:2}, that is, $\max_{A \in \cW_{n_0}(\xi_\infty)} d(A, \cW_{n_0}(\xi_k)) > \varepsilon$, happens for infinitely many values of $k$, then again, by passing to a subsequence, we assume that this happens for all $k$.
We now argue in a similar manner to the case of option \ref{option:1}. 
As above, letting $A^{(k)} \in \cW_{n_0}(\xi_\infty)$ be such that $d(A^{(k)}, \cW_{n_0}(\xi_k)) > \varepsilon$, we find a $\delta > 0$ and a sequence $p_{k}$ in the unit ball of $\in M_{n_0}(\bC_1 \langle z \rangle)$ such that
\[
\left \| p_{k}(\xi_\infty) \right\| \geq \left\|p_{k}\left(A^{(k)}\right)\right\|  > \left\|p_{k}(B)\right\| + \delta
\]
for all $k$ and all $B \in \cW(\xi_k)$. 
Thus, $\|p_k(\xi_\infty)\| \geq \|p_k(\xi_k)\| + \delta$ for all $k$. 
Letting $p$ be a limit point of the sequence $(p_{k})_k$, we obtain
\[
\left \| p(\xi_\infty) \right\|  \geq \left\|p(\xi_k)\right\| + \delta/2,
\]
for all $k$ in an infinite subset $K \subseteq \bN$, which is a contradiction to \ref{it:linmatpolys}.

We now prove that \ref{it:matrange} $\Rightarrow$ \ref{it:linmatpolys}.
First, let us assume that there is some compact neighborhood of $0$ contained in $\operatorname{int}\cW_1(\xi_\infty)$, which implies that there is a neighborhood contained in all $\cW(\xi_k)$ for $k$ sufficiently large.
Fix $k \in \bN$ and a nonzero $p \in M_k(\bC_1\langle z \rangle)$.
Fix also some $\varepsilon > 0$. Let $Y$ be the compression of $\xi_\infty$ to some $n$-dimensional subspace such that $\|p(Y)\| > \|p(\xi_\infty)\| - \varepsilon$.
Since $\cW_n(\xi_k) \to \cW_n(\xi_\infty)$, we have $d_H(\cW_n(\xi_k), \cW_n(\xi_\infty)) < \frac{\varepsilon}{\|p\|}$ for all sufficiently large $k$.
For every such $k$, we can find $X \in \cW_n(\xi_k)$ such that $d(X,Y) < \frac{\varepsilon}{\|p\|}$, and therefore
\[
\|p(\xi_k)\| \geq \|p(X)\| \geq \|p(Y)\| - \|p(X)-p(Y)\| > \|p(\xi_\infty)\| - 2\varepsilon.
\]
Thus, we proved that $\liminf_{k \to \infty}\|p(\xi_k)\| \geq \|p(\xi_\infty)\|$ (we have not yet made use of the assumption on the interior).

In order to prove that $\limsup_{k \to \infty}\|p(\xi_k)\| \leq \|p(\xi_\infty)\|$, we assume to the contrary that there is some $n$, some $p \in M_n(\bC_1\langle z \rangle)$ and some $\varepsilon > 0$, such that for infinitely many $k$,
\[
\|p(\xi_k)\| > \|p(\xi_\infty)\| + \varepsilon.
\]
We now we make use of the assumption that $0 \in \operatorname{int}\cW_1(\xi_\infty)$.
Since $\cW_{2n}(\xi_k) \to \cW_{2n}(\xi_\infty)$, we can use Lemma \ref{lem:close_implies_dilation} to find a $c$ very close to $1$ such that $\cW_{2n}(\xi_k) \subseteq c\cW_{2n}(\xi_\infty)$ for all large enough $k$.
By \cite[Theorem 2.5]{Zalar} (combined with
\cite[Propositions 3.1 and 3.3, Lemma 3.4]{DDSS17} to make the transition from spectrahedra to matrix ranges), the map $c \xi_{\infty,i} \mapsto \xi_{k,i}$ ($i=1,\ldots, d$) extends to a $2n$-positive and unital map between the respective operator systems.
But a unital $2n$-positive map is $n$-contractive (see, e.g., \cite[Proposition 3.2]{PauBook}).
We therefore have for all large $k$,
\[
\|p(\xi_k)\| \leq \|p(c\xi_\infty)\| \leq \|p(\xi_\infty)\| + (c-1)\|p\| < \|p(\xi_k)\| - \varepsilon + (c-1)\|p\|.
\]
Since $c \searrow 1$ as $k \to \infty$, this is a contradiction.
Therefore, $\limsup_{k \to \infty}\|p(\xi_k)\| \leq \|p(\xi_\infty)\|$.

Finally, we show that \ref{it:matrange} $\Rightarrow$ \ref{it:linmatpolys} holds also without
the assumption $0 \in \operatorname{int}\cW_1(\xi_\infty)$.
If $\cW_1(\xi_\infty) \subseteq \bR^d$ has no interior, then there exists an affine subspace of minimal dimension in $\bR^d$ containing $\cW_1(\xi_\infty)$.
We shall treat the case that this subspace has dimension $m$, where $0<m<d$; the case $m=d$ was treated above (up to a shift), and the case $m=0$ is easy and omitted.
By an affine change of variables (see
\cite[Section 3.1]{PSS18}), we may assume that $\cW_1(\xi_\infty) \subseteq \bR^m \subseteq \bR^d$ (the subspace consisting of tuples $(x_1, \ldots x_d)$ such that $x_i = 0$ for $i>m$), that $\xi_\infty = (\xi_{\infty,1}, \ldots, \xi_{\infty,m}, 0, \ldots, 0)$, and that $0 \in \operatorname{int}\cW_1(\widetilde{\xi_\infty})$, where $\widetilde{\xi_\infty} = (\xi_{\infty,1}, \ldots, \xi_{\infty,m})$.
Let also $\widetilde{\xi_k} = (\xi_{k,1}, \ldots, \xi_{k,m})$.
Since $\cW(\widetilde{\xi_k})$ is the projection of $\cW(\xi_k)$ onto the first $m$ variables, we get from \ref{it:matrange} that $\cW_n(\widetilde{\xi_k}) \to \cW_n(\widetilde{\xi_\infty})$ for all $n$.
By the previous paragraph, this implies that $\|p(\widetilde{\xi_k})\| \to \|p(\widetilde{\xi_\infty})\|$ for all $p \in M_n(\bC_1\langle \widetilde{z} \rangle)$.
But since $\cW_1(\xi_\infty) \subseteq \bR^m$ and $\cW_1(\xi_k) \to \cW_1(\xi_\infty)$, we also get that $\|\xi_{k,j}\| \xrightarrow{k\to \infty} 0$ for all $j = m+1, \ldots, d$ (because the norm of a selfadjoint operator is attained by a state).
From this it follows that $\|p({\xi_k})\| \to \|p({\xi_\infty})\|$ for all $p \in M_n(\bC_1\langle {z} \rangle)$, and the proof of \ref{it:matrange} $\Rightarrow$ \ref{it:linmatpolys} is complete.
\end{proof}

Clearly \ref{it:linmatpolys} cannot imply \ref{it:polys}, because there exist completely isometric operator systems that generate nonisomorphic C*-algebras.
However, using Arveson's boundary theory, we can show that if the tuples are in some sense rigid, then the implication \ref{it:linmatpolys} $\Rightarrow$ \ref{it:polys} does hold.
Although we do not require it in the sequel, we record this interesting fact in a special case.

\begin{theorem}\label{thm:convergences_unitary}
Let $\cT$, $\xi_t$, and $\cA_t$ be as in Theorem \ref{thm:convergences}.
If all the $\xi_{t,i}$ (for all $t \in \cT$ and $i =1, \ldots, d$) are unitary, then \ref{it:linmatpolys} implies \ref{it:polys} (and therefore also \ref{it:matpolys}).
\end{theorem}
\begin{proof}
It suffices to prove \ref{it:linmatpolys} $\Rightarrow$ \ref{it:polys}, and to consider the case $\cT = \hat{\bN}$ and $t_0 = \infty$.
Let $\prod_n \cA_n$ be the C*-algebra of all bounded sequences $(a_n)_n$ with $a_n \in \cA_n$ endowed with the natural C*-algebra structure, and let $\sum_n \cA_n$ be the ideal of sequences $(a_n)_n$ such that $\lim_n a_n = 0$.
Let $\cB = \prod_n \cA_n / \sum_n \cA_n$ be the
quotient algebra. 
For every $i=1, \ldots, d$, let $u_i$ denote the image of the sequence $(\xi_{n,i})_n$ in $\cB$.
Then clearly $u_1, \ldots, u_d$ is a tuple of unitaries and it satisfies $\|p(u)\| = \limsup_{n\to \infty}\|p(\xi_n)\|$ for every matrix valued $*$-polynomial $p$.
If \ref{it:linmatpolys} holds, then the map $u_i \mapsto \xi_{\infty,i}$ extends to a completely isometric UCP map from the operator system generated by $u$ in $\cB$ to the operator system generated by $\xi_\infty$ in $\cA_\infty$.
By \cite[Corollary 2.2.8]{Arv69}, the Shilov boundary of an operator system generated by unitaries, relative to the C*-algebra they generate, is trivial.
Therefore, by \cite[Theorem 2.2.5]{Arv69}, the map $u_i \mapsto \xi_{\infty,i}$ is actually implemented by a $*$-isomorphism from $C^*(u) \subseteq \cB$ onto $\cA_\infty$.
This implies $\|p(\xi_\infty)\| = \|p(u)\| = \limsup_{n\to \infty}\|p(\xi_n)\|$ for every $*$-polynomial $p$.

In order to show \ref{it:polys}, we need to show that the limit superior is actually a limit. However, for every subsequence $(\xi_{n_k})_k$, we can repeat the above argument with $\cB = \prod_k \cA_{n_k} / \sum_k \cA_{n_k}$ to find that $\|p(\xi_\infty)\| = \limsup_{k\to \infty}\|p(\xi_{n_k})\|$. This shows that the limit exists, and that $\|p(\xi_\infty)\| = \lim_{n\to \infty}\|p(\xi_n)\|$ for every $*$-polynomial $p$, proving \ref{it:polys}.
\end{proof}

\begin{remark}
The use of Arveson's boundary theory (perhaps without explicitly referring to it) to pass from complete isometries to $*$-isomorphisms, has appeared in the literature under the name ``linearization trick''; see \cite{Pis19}. 
It was used by Haagerup and Thorbj{\o}rnsen \cite[Theorem 2.2]{HT05} to pass from strong convergence on linear matrix valued polynomials to strong convergence on arbitrary polynomials. 
Haagerup and Thorbj{\o}rnsen, in turn, were inspired by Pisier's simple proof of Kirchberg's theorem \cite{Pis96}. 
Our use of the
quotient $\prod_n \cA_n / \sum_n \cA_n$ to pass from asymptotic expressions to complete isometries is inspired by the proof of \cite[Proposition 7.3]{HT05} (though this idea surely appears elsewhere as well). 
If $\cW(\xi_t)$ all contain some neighborhood of the origin, and if the convergence of the levels $\cW_n(\xi_t) \to \cW_n(\xi_{t_0})$ is uniform in $n$, then an alternative proof can be given using the methods of \cite[Section 4]{GS20}.
\end{remark}

\section{Random matrix ranges}\label{sec:matrix-range-free}

We begin by making a general observation which follows from the results of the previous section.

\begin{theorem}\label{thm:random_matrange}
Let $T^N = (T_1^N, \ldots, T_d^N)$ be a random matrix ensemble and $t = (t_1, \ldots, t_d)$ a $d$-tuple of operators on a Hilbert space.
Suppose that
\[
\lim_{N\to\infty}\|p(T^{N})\|=\|p(t)\|
\]
almost surely, for all $p \in \bC\langle z, z^* \rangle$. Then for all $n \in \bN$, $d_H(\cW_n(T^{N}),\cW_n(t)) \xrightarrow{N \to \infty} 0$, almost surely.
\end{theorem}
\begin{proof}
For every element $\omega$ in the underlying probability space for which $\lim_{N\to\infty}\|p(T^{N}(\omega))\|=\|p(t)\|$, we apply \ref{it:polys} $\Rightarrow$ \ref{it:matrange} of Theorem \ref{thm:convergences}.
\end{proof}

When combined with known results on strong convergence of matrix ensembles, the above theorem immediately implies convergence of matrix ranges of these matrix ensembles.

We say that a family $X^N = (X^N_1, \ldots, X^N_d)$ of random complex $N \times N$ matrices is a \emph{Wigner ensemble} if:
\begin{itemize}
\item the matrix $X^N_k$ is selfadjoint for all $k$ and $N$;
\item for all $N\in\mathbb N$, the real valued random variables $(X^N_k)_{ii}$, $\re (X^N_k)_{ij}$ and $\im (X^N_k)_{ij}$ ($k=1, \ldots, d$; $1\leq i< j \leq N$) are independent with expectation $0$;
\item $\sqrt{N} (X^N_k)_{ii}$, ($k=1, \ldots, d$; $i=1, \ldots, N$; $N=1, 2, \ldots$) are identically distributed;
\item $\sqrt{2N} \re (X^N_k)_{ij}$ and $\sqrt{2N}  \im (X^N_k)_{ij}$ ($k=1, \ldots, d$; $1\leq i < j \leq N$; $N=1, 2, \ldots$) are identically distributed with variance $1$.
\end{itemize}
In particular, we assume a Wigner ensemble to fulfill $\bE(|(X^N_{k})_{ij}|^2) = 1/N$ for $i\neq j$. 

A \emph{free semicircular family} is a $d$-tuple $s = (s_1, \ldots, s_d)$ of selfadjoint elements in a C*-probability space $(\cA, \tau, \|\cdot\|)$ such that:
\begin{itemize}
\item $s_1, \ldots, s_d$ are \emph{freely independent} (meaning that whenever $x_{k} \in \alg(1,s_{i_k})$, $i_k \neq i_{k+1}$ and $\tau(x_{k}) = 0$ for all $k$, then $\tau(x_{1} \cdots x_{n}) = 0$);
\item every $s_i$ is a \emph{semicircular element} (meaning that $\tau(s_i^k) = \frac{1}{2\pi} \int_{-2}^2 t^k  \sqrt{4-t^2} \mathrm{d}t$ for all $i$ and all $k$).
\end{itemize}

It is well known that Wigner ensembles converge in distribution to free semicircular families \cite{VDNBook}. It is therefore natural to guess that if the matrix ranges $\cW(X^N)$ of a Wigner ensemble converge to anything, then they converge to the matrix range $\cW(s)$ of a free semicircular family $s$. This is indeed the case.

\begin{theorem}\label{thm:wigner}
Let $X^N = (X^N_1, \ldots, X^N_d)$ be a Wigner ensemble and assume that $\bE(|(X^N_{k})_{ij}|^4) < \infty$ for all $i,j$. Let $s = (s_1, \ldots, s_d)$ be a semicircular $d$-tuple in a C*-probability space. Then the matrix range $\cW(X^N)$ converges levelwise in the Hausdorff metric to $\cW(s)$, almost surely, that is, for all $n$,
\[
\lim_{N \to \infty} d_H(\cW_n(X^N),\cW_n(s)) = 0 \,\, , \,\, a.s.
\]
\end{theorem}
\begin{proof}
The theorem follows immediately from Theorem \ref{thm:random_matrange} together with the remarkable result \cite{And13} of Anderson that the limit $\lim_{N\to\infty}\|p(X^{N})\|=\|p(s)\|$ holds almost surely for every $*$-polynomial $p$.
\end{proof}

\begin{remark}
Anderson's result was proved first for the GUE ensemble by Haagerup and Thorbj{\o}rnsen \cite{HT05}.
The strong convergence $\lim_{N\to\infty}\|p(X^{N})\|=\|p(s)\|$ was proved for the GOE and GSE cases by Schultz \cite{Sch05}, and for certain symmetric Wigner ensembles ``that satisfy a Poincar\'{e} inequality'' by Capitaine and Donati-Martin \cite{CD-M07}, and thus the convergence of matrix ranges also holds in these cases.
\end{remark}

Above, we covered perhaps the most natural model by which one may generate a random $d$-tuple of matrices, which is by independently sampling the entries of $d$ selfadjoints from a fixed distribution. Another very natural way to generate a random $d$-tuple of matrices is to sample $d$ independent unitaries from the Haar measure on the compact group $\cU_N$ of $N \times N$ unitaries.

Recall that a \emph{free tuple of Haar unitaries} is a $d$-tuple $u = (u_1, \ldots, u_n)$ of unitaries in a C*-probability space $(\cA, \tau, \|\cdot\|)$ which are freely independent and satisfy $\tau(u_i^k) = 0$ for all $i$ and all $k \in \bZ \setminus \{0\}$.

\begin{theorem}\label{thm:haar}
Let $U^N = (U^N_1, \ldots, U^N_d)$ be an ensemble of $d$ independent $N \times N$ unitaries distributed according to the Haar measure in $\cU_N$, and let $u = (u_1, \ldots, u_d)$ be a free tuple of Haar unitaries in a C*-probability space. Then the matrix range $\cW(U^N)$ converges levelwise in the Hausdorff metric to $\cW(u)$, almost surely, that is, for all $n$,
\[
\lim_{N \to \infty} d_H\left(\cW_n(U^N),\cW_n(u)\right) = 0 \,\, , \,\, a.s.
\]
\end{theorem}
\begin{proof}
This follows from Theorem \ref{thm:random_matrange} together with another remarkable result, this time due to Collins and Male \cite{CM14}, that $\lim_{N\to\infty}\|p(U^{N})\|=\|p(u)\|$, almost surely, for every $*$-polynomial $p$.
\end{proof}

\section{Description of the limiting matrix ranges}\label{sec:description}

In the previous section, we found that $\cW(X^N)$ converges to $\cW(s)$ for independent Wigner ensembles of selfadjoints, and that $\cW(U^N)$ converges to $\cW(u)$ for independent Haar distributed ensembles of unitaries. In this section, we will try to identify $\cW(s)$ and $\cW(u)$ in order to understand what the limiting matrix ranges ``look like''. The description we give will be based on the work of Lehner \cite{Leh99}.

We will describe a matrix ball $\cL \subseteq \bM^d_{\sa}$ that we will call the {\em Lehner ball}, that satisfies $\cW(\frac{s}{2}) = \cL^\bullet$.
The reason that we call it a ``matrix ball'' is that it is a matrix convex set that has at its first level the closed Euclidean ball of radius one.
In the next section we will show that $\cL$ is not any one of the usual suspects, that is $\cL \notin \{\fB, \fB^\bullet, \fD, \Wmin{}(\ol{\bB}_d), \Wmax{}(\ol{\bB}_d)\}$.
Equivalently, $\cW(s/2)$ is not any one of the above familiar matrix balls.

We need to recall a concrete realization of the semicircular tuple
$s$. 
Let $\cF(\bC^d) = \bigoplus_{n=0}^\infty (\bC^d)^{\otimes n}$ be the full Fock space over $\bC^d$. Denote by $\Omega$ the vector $\Omega :=1 \in \bC = (\bC^d)^{\otimes 0}$ in the first summand. Let $\{e_1, \ldots, e_d\}$ be the standard orthonormal basis of $\bC^d$, and for $i=1, \ldots, d$ define the shift operator by its action on elementary tensors:
\[
\ell_i (x_1 \otimes \cdots \otimes x_n) = e_i \otimes x_1 \otimes \cdots \otimes x_n \quad , \quad n = 1, 2, \ldots
\]
and $\ell_i(\Omega) = e_i$. It is straightforward to check that $\ell_1, \ldots, \ell_d$ are isometries with pairwise orthogonal ranges. On the C*-algebra $\cT(\bC^d) = C^*(\ell_1, \ldots, \ell_d)$ we define the state $\tau(a) = \langle a \Omega, \Omega \rangle$. Voiculescu proved \cite{Voi85} that the tuple $s = (s_1, \ldots, s_d)$ defined by $s_i = \ell_i+ \ell_i^*$ ($i=1, \ldots, d$) is a free $d$-tuple of semicirculars.

We define the \emph{Lehner ball} to be the set
\[
\cL := \cD_{s/2}^{\sa} = \cW(s/2)^\bullet = \left\{X \in \bM^d_{\sa} : \sum X_i \otimes \frac{s_i}{2} \leq I\right\}.
\]
Since $\ell_i \mapsto -\ell_i$ gives rise to an automorphism of $\cT(\bC^d)$ that maps $s$ to $-s$, we have
\[
\cL = \left\{X \in \bM^d_{\sa} : \left\|\sum X_i \otimes s_i \right\| \leq 2\right\}.
\]
By \cite[Theorem 1.3]{Leh99}, we have the alternative description
\begin{align}
  \label{eq:LehnerBallFormula}
  \cL_n = \left\{X \in (M_n^d)_{\sa} : \inf_{Z > 0} \left\| Z + \sum X_i Z^{-1} X_i \right\| \leq 2 \right\},
\end{align}
where the infimum is taken over all positive definite (invertible) $n \times n$ matrices $Z$.

\begin{proposition}\label{prop:LehnerBalllevel1}
The first level of the Lehner ball is given by $\cL_1 = \ol{\bB}_d$. Consequently, the numerical range of $s$ is given by $\cW_1(s) = 2 \ol{\bB}_d$.
\end{proposition}
\begin{proof}
Fix $x \in \bR^d$. By solving an elementary exercise in calculus we find that the function $f\colon (0,\infty) \to (0,\infty)$ given by
\[
f(t)  = \left| t + \sum x_i t^{-1} x_i \right| = \left| t + \frac{\|x\|_2^2}{t} \right|
\]
attains its minimum at $t = \|x\|$. Together with formula \eqref{eq:LehnerBallFormula} above, this means that $\cL_1 = \ol{\bB}_d$. Since $\cW(\frac{s}{2}) = (\cD_\frac{s}{2}^{\sa})^\bullet = \cL^\bullet$ (see \cite[Proposition 3.3]{DDSS17}), we see that $\cW_1(s) =2 \cdot (\ol{\bB}_d)' = 2 \ol{\bB}_d$.

Alternatively, we can use the first description and define $\ell_x:=\sum x_i\ell_i$ and $s_x:=\ell_x + \ell_x^*$ for $x\in \mathbb R^d$. Now it is easy to see that $s_x$ is unitarily equivalent to $\|x\| s_1$, therefore $\|\sum x_i\otimes s_i\|=\|s_x\|=\|x\|\cdot\|s_1\|=2\|x\|$.
\end{proof}

In \cite{CGLZ14}, the numerical range of large random complex matrices was investigated. For example, it was shown that if $G^N$ denotes the Ginibre ensemble, then the numerical range $W(G^N)$ converges almost surely in the Hausdorff metric to the closed unit disc of radius $\sqrt{2}$.
The definition of the numerical range used there was $W(A) = \{\langle Ah, h \rangle : \|h\| =1\}$, but it is not hard to see that this is equivalent to the notion we are using, that is $W(A) = \cW_1(A)$ (for the proof one needs the Hausdorff-Toeplitz theorem, that is, the fact that $W(A)$ is convex, together with the fact that vector states are the pure states on $M_n$).
Moreover, the Ginibre ensemble is obtained as a combination $G^N = \frac{1}{\sqrt{2}} (X_1^N + \mathrm{i} X_2^N)$ where $X_1^N, X_2^N$ are independent GUE matrices.
Thus, the above mentioned result on the convergence of the numerical range $W(G^N)$ is equivalent to the statement that $\cW_1(X^N) \xrightarrow{n\to \infty} 2 \ol{\bB}_d$ when $d=2$.
Our Theorem \ref{thm:wigner} yields the following significant generalization for $d>2$ of the result from \cite{CGLZ14}.

\begin{corollary}\label{cor:randomNumRange}
Let $X^N = (X^N_1, \ldots, X^N_d)$ be a Wigner ensemble such that $\bE(|(X^N_{k})_{ij}|^4) < \infty$ for all $i,j$. Then the numerical range $\cW_1(X^N)$ converges in the Hausdorff metric to the ball $2 \ol{\bB}_d$ almost surely.
\end{corollary}
\begin{proof}
This follows from the Proposition \ref{prop:LehnerBalllevel1} together with Theorem \ref{thm:wigner}.
\end{proof}

What do the numerical ranges $\cW_1(U^N)$ converge to? As a particular case of Theorem \ref{thm:haar}, we know that almost surely they converge to $\cW_1(u)$, but we don't have a very clear description of $\cW_1(u)$. Its first level is certainly not a ball and not a polydisc.

\begin{proposition}\label{prop:LehnerPolydisclevel1}
Let $u$ be a $d$-tuple of free Haar unitaries. Define
\[
E := \left\{x \in \bR^d : \inf_{t>0} \left|\sum \sqrt{t^2 + x_i^2} - (d-1)t \right| \leq 1 \right\},
\]
and
\[
Q = \{z \in \bC^d : (|z_1|, \ldots, |z_d|) \in E\}.
\]
Then the numerical range $\cW_1(u)$ is given by the polar dual of $Q$,
\[
\cW_1(u) = Q' = \left\{w \in \bC^d : \re \sum w_i z_i \leq 1 \textrm{ for all } z \in Q\right\}.
\]
In particular, $0 \in \operatorname{int}\cW_1(u)$.
\end{proposition}
\begin{proof}
We begin by examining
\[
\cD_u = \left\{X \in \bM^d : \re \sum u_i \otimes X_i \leq I \right\}.
\]
We note that once we show that $\cD_u(1)$ is bounded,
it follows that 
$0 \in \operatorname{int}(\cW_1(u))$
\cite[Lemma 3.4]{DDSS17} 
and therefore also that $\cW(u) = (\cD_u)^\circ$  
 \cite[Lemma 3.2 and Proposition 3.3]{DDSS17}. 
It is then not hard to see that $\cW_1(u) = (\cD_u(1))'$. It therefore remains to show that $\cD_u(1) = Q$ as described above, and to explain why this set is bounded.

Since $u_i \mapsto - u_i$ induces a $*$-automorphism on $C^*(u)$, we have that, for all $X \in \bM^d$, $\re \sum u_i \otimes X_i \leq I$ if and only if $\left\| \re \sum u_i \otimes X_i \right\| \leq 1$. Thus,
\[
\cD_u = \left\{X \in \bM^d : \left\|\re \sum u_i \otimes X_i \right\| \leq 1 \right\},
\]
and, in particular,
\[
\cD_u(1) = \left\{x \in \bC^d : \frac{1}{2}\left| \sum u_i \otimes x_i + \sum u_i^* \otimes \ol{x}_i \right| \leq 1 \right\}.
\]
Again, since $u_i \mapsto \mathrm e^{\mathrm i\theta_i} u_i$ induces a $*$-automorphism on $C^*(u)$, we find that $x \in \cD_u(1)$ if and only if $|x| := (|x_1|, \ldots, |x_d|) \in \cD_u(1)$. But by \cite[Theorem 1.1]{Leh99}, after a little bit of rearranging, we find that for $x \in \bR^d$,
\[
\frac{1}{2}\left| \sum u_i \otimes x_i + \sum u_i^* \otimes x_i \right| = \inf_{t>0} \left|\sum \sqrt{t^2 + x_i^2} - (d-1)t \right|.
\]
This establishes that $Q = \cD_u(1)$. Since the right hand side of the above equation is always greater than $\|x\|_\infty$, we obtain that $\cD_u(1)$ is bounded.
\end{proof}

\section{Comparison of $\cW(s)$ with other popular matrix balls}

In this section, we compare $\cW(s/2)$ and $\cL$ to other matrix balls that have been considered: $\Wmin{}(\ol{\bB}_d)$, $\fB$, $\fD$, $\fB^\bullet$ and $\Wmax{}(\ol{\bB}_d)$, and we show that they are different from all of them.
\begin{lemma}
Let $a_1,\ldots, a_d$ be selfadjoint. Then $\|\sum a_i\otimes s_i\| \leq 2\|\sum a_i^2\|^{\frac{1}{2}}$.
\end{lemma}
\begin{proof}
Since we can write $s_i=\ell_i^*+\ell_i$, where $\ell_i$ are the free shifts on full Fock space, we have
\[
\left\|\sum a_i\otimes \ell_i\right\|^2=\left\|\sum a_ia_j\otimes \ell_i^*\ell_j\right\|=\left\|\sum  a_i^2\otimes 1\right\|=\left\|\sum a_i^2\right\|.
\]
Therefore
\[
\left\|\sum a_i\otimes s_i\right\|\leq \left\|\sum a_i\otimes \ell_i\right\| + \left\|\sum a_i\otimes \ell_i^*\right\| = 2\left\|\sum a_i^2\right\|^{\frac{1}{2}}.\qedhere
\]
\end{proof}
The above lemma immediately implies that $\fB\subset \cW(s/2)^{\bullet} = \cL$ or, by duality, $\cW(s)\subset 2\fB^{\bullet}$.
In particular, $\cW(s) \neq \Wmax{}(2 \ol{\bB}_d)$, and $\cL \neq \Wmin{}( \ol{\bB}_d)$.
That the inclusion  $\fB \subset \cL$ (and therefore $\cW(s)\subset 2\fB^{\bullet}$) is strict, follows easily from the following lemma.

\begin{lemma}
  There exist selfadjoint $2\times 2$ matrices $b_1, b_2$ such that
  \[\|b_1 \otimes s_1 + b_2 \otimes s_2\| \leq 2 \text{ while } \|b_1^2 + b_2^2\| > 1.\]
\end{lemma}
\begin{proof}
Consider the matrices
\[
a_1 = \begin{pmatrix} \sqrt{t} & 1 \\ 1 & \frac{1}{\sqrt{t}} \end{pmatrix} \,\, , \,\, a_2 = \begin{pmatrix} \sqrt{t} & -1 \\ -1 & \frac{1}{\sqrt{t}} \end{pmatrix} .
\]
A direct calculation shows that $\| a_1^2 + a_2^2\|^{1/2} = \sqrt{2(1+t)}$.
On the other hand, by Equation \eqref{eq:LehnerBallFormula} we can bound the norm of $\|a_1 \otimes s_1 + a_2 \otimes s_2 \|$ by
\[
\left\|a_1 \otimes s_1 + a_2 \otimes s_2 \right\| \leq \left\| D + a_1 D a_1 + a_2 D a_2  \right\|
\]
where $D$ is the diagonal matrix
\[
D = \begin{pmatrix} \lambda & 0 \\ 0 & \mu \end{pmatrix} .
\]
We find that
\[
\left\|a_1 \otimes s_1 + a_2 \otimes s_2 \right\| \leq \max\{\lambda + 2(\lambda^{-1}t+\mu^{-1}), \mu + 2(\lambda^{-1} + \mu^{-1} t^{-1}) \},
\]
for any positive $\lambda, \mu$. Choosing $\lambda = 3, \mu = 7$ and $t = 7$, we find that
\[
\|a_1^2 + a_2^2 \|^{1/2} = 4,
\]
while
\[
\|a_1 \otimes s_1 + a_2 \otimes s_2 \| =:\delta < 8.
\]
Dividing $a$  by $\delta/2$ we obtain $b_1, b_2$ such that $\|b_1 \otimes s_1 + b_2 \otimes s_2\| \leq 2$ while $\|b_1^2 + b_2^2\| > 1$.
\end{proof}

We also have $\fB\subset \cW(s/2)=\cL^\bullet$. Indeed, for any row contraction $a$, i.e., $\sum a_i a_i^*\leq 1$, there is a UCP map $\Phi$ sending $\ell_i$ to $a_i$ \cite{Pop91}. If the $a_i$ are also selfadjoint, i.e., if $a\in\fB$, we find that $\Phi(s/2)=a$.
Due to the next lemma, $\|s/2\|>1$, so the containment is strict.

\begin{lemma}
   \[\|s\|=\left\| \sum s_i^2\right\|^{\frac{1}{2}}=1+\sqrt{d}.\] 
\end{lemma}

\begin{proof}
Since $s_i = \ell_i + \ell_i^*$, we have
\[
\sum s_i^2 = (d+1)I - P_\Omega + \sum (\ell_i^2 + \ell_i^{*2}) 
\]
where we write $P_\Omega$ for the vacuum projection.
Now, the tuple $(P_\Omega,\ell_1, \ldots, \ell_d)$ is $*$-isomorphic to the compression of $(P_\Omega,\ell_1^2, \ldots, \ell_d^2)$ to $\mathcal F(\operatorname{span}(e_1\otimes e_1,\ldots,e_d\otimes e_d))$.
Hence
\[
  \left\| \sum s_i^2\right\| \geq \left\|(d+1)I - P_\Omega + \sum (\ell_i + \ell_i^*)\right\|.
\]
The last expression can be computed using the formula (1.6) in \cite[Theorem 1.8]{Leh99}, which reads in our case as
\[
  \left\|(d+1)I - P_\Omega + \sum (\ell_i + \ell_i^*)\right\| = \inf_{t>1} \left|t + d + \frac{d}{t-1}\right|.
\]
The minimum is easily found to be attained at $t = 1+ \sqrt{d}$, and it is equal to $(1+\sqrt{d})^2$.
Equality holds in the expression for the norm, because $\|(d+1)I - P_\Omega + \sum (\ell_i^2 + \ell_i^{*2})\|\leq d+1 + 2\|\sum \ell_i^2\|=d+1+2\sqrt{d}=(1+\sqrt{d})^2$.
\end{proof}

By applying $\sum \ol{s}_i\otimes s_i$ to $x_N\otimes x_N$, where $x_N:=\frac{1}{\sqrt{N}}\sum_{k=1}^{N} e_1^{\otimes k}$ is an approximate eigenvector for $s_1$ with approximate eigenvalue 2,
we find 
that $\|\sum \ol{s}_i\otimes s_i\|>4$ for
$d\geq2$. 
This means that $\cW(s/2)\nsubseteq \fD$.
By selfduality of $\fD$, we also get $\fD \nsubseteq \cL$.

To conclude, let us record the result of our comparison of the Lehner ball $\cL$ and its dual $\cW(s/2)$ with the other matrix balls.
\begin{theorem}
  Whenever $d\geq 2$, we have the following relations.
\[
\Wmin{}(\ol{\bB}_d) \subset \fB \subsetneq \cL,\cW(s/2) \subsetneq \fB^\bullet \subset \Wmax{}(\ol{\bB}_d),
\]
and
\[
\cW(s/2) \nsubseteq \fD \nsubseteq \cL.
\]
\end{theorem}


\bibliographystyle{myplainurl}

 \linespread{1.25}

\bibliography{rMR}

\providecommand{\bysame}{\leavevmode\hbox to3em{\hrulefill}\thinspace}
\providecommand{\MR}{\relax\ifhmode\unskip\space\fi MR }
\providecommand{\MRhref}[2]{%
  \href{http://www.ams.org/mathscinet-getitem?mr=#1}{#2}
}
\providecommand{\href}[2]{#2}
\begin{thebibliography}{10}

\bibitem{And13}
G.~W. Anderson, \emph{Convergence of the largest singular value of a polynomial
  in independent {W}igner matrices}, Ann. Probab. \textbf{41} (2013), no.~3B,
  2103--2181, \href {https://doi.org/10.1214/11-AOP739}
  {\path{doi:10.1214/11-AOP739}}.

\bibitem{Arv69}
W.~Arveson, \emph{Subalgebras of {$C^{\ast} $}-algebras}, Acta Math.
  \textbf{123} (1969), 141--224, \href {https://doi.org/10.1007/BF02392388}
  {\path{doi:10.1007/BF02392388}}.

\bibitem{Arv72}
W.~Arveson, \emph{Subalgebras of {$C^{\ast} $}-algebras. {II}}, Acta Math.
  \textbf{128} (1972), no.~3-4, 271--308, \href
  {https://doi.org/10.1007/BF02392166} {\path{doi:10.1007/BF02392166}}.

\bibitem{CD-M07}
M.~Capitaine and C.~Donati-Martin, \emph{Strong asymptotic freeness for
  {W}igner and {W}ishart matrices}, Indiana Univ. Math. J. \textbf{56} (2007),
  no.~2, 767--803, \href {https://doi.org/10.1512/iumj.2007.56.2886}
  {\path{doi:10.1512/iumj.2007.56.2886}}.

\bibitem{CGLZ14}
B.~Collins, P.~Gawron, A.~E. Litvak, and K.~\.{Z}yczkowski, \emph{Numerical
  range for random matrices}, J. Math. Anal. Appl. \textbf{418} (2014), no.~1,
  516--533, \href {https://doi.org/10.1016/j.jmaa.2014.03.072}
  {\path{doi:10.1016/j.jmaa.2014.03.072}}.

\bibitem{CM14}
B.~Collins and C.~Male, \emph{The strong asymptotic freeness of {H}aar and
  deterministic matrices}, Ann. Sci. \'{E}c. Norm. Sup\'{e}r. (4) \textbf{47}
  (2014), no.~1, 147--163, \href {https://doi.org/10.24033/asens.2211}
  {\path{doi:10.24033/asens.2211}}.

\bibitem{DDSS17}
K.~R. Davidson, A.~Dor-On, O.~M. Shalit, and B.~Solel, \emph{Dilations,
  inclusions of matrix convex sets, and completely positive maps}, Int. Math.
  Res. Not. IMRN \textbf{13} (2017), 4069--4130, \href
  {https://doi.org/10.1093/imrn/rnw140} {\path{doi:10.1093/imrn/rnw140}}.

\bibitem{DixmierBook}
J.~Dixmier, \emph{{$C\sp*$}-algebras}, North-Holland Publishing Co.,
  Amsterdam-New York-Oxford, 1977, Translated from the French by Francis
  Jellett, North-Holland Mathematical Library, Vol. 15.

\bibitem{EW97}
E.~G. Effros and S.~Winkler, \emph{Matrix convexity: operator analogues of the
  bipolar and {H}ahn-{B}anach theorems}, J. Funct. Anal. \textbf{144} (1997),
  no.~1, 117--152, \href {https://doi.org/10.1006/jfan.1996.2958}
  {\path{doi:10.1006/jfan.1996.2958}}.

\bibitem{Eve18}
E.~Evert, \emph{Matrix convex sets without absolute extreme points}, Linear
  Algebra Appl. \textbf{537} (2018), 287--301, \href
  {https://doi.org/10.1016/j.laa.2017.09.033}
  {\path{doi:10.1016/j.laa.2017.09.033}}.

\bibitem{FNT17}
T.~Fritz, T.~Netzer, and A.~Thom, \emph{Spectrahedral containment and operator
  systems with finite-dimensional realization}, SIAM J. Appl. Algebra Geom.
  \textbf{1} (2017), no.~1, 556--574, \href
  {https://doi.org/10.1137/16M1100642} {\path{doi:10.1137/16M1100642}}.

\bibitem{GS20}
M.~Gerhold and O.~M. Shalit, \emph{Dilations of $q$-commuting unitaries}, Int.
  Math. Res. Not. IMRN (2020), 26 pp., rnaa093, \href
  {https://doi.org/10.1093/imrn/rnaa093} {\path{doi:10.1093/imrn/rnaa093}}.

\bibitem{HT05}
U.~Haagerup and S.~Thorbj\o~rnsen, \emph{A new application of random matrices:
  {${\rm Ext}(C^*_{\rm red}(F_2))$} is not a group}, Ann. of Math. (2)
  \textbf{162} (2005), no.~2, 711--775, \href
  {https://doi.org/10.4007/annals.2005.162.711}
  {\path{doi:10.4007/annals.2005.162.711}}.

\bibitem{HKMS19}
J.~W. Helton, I.~Klep, S.~McCullough, and M.~Schweighofer, \emph{Dilations,
  linear matrix inequalities, the matrix cube problem and beta distributions},
  Mem. Amer. Math. Soc. \textbf{257} (2019), no.~1232, vi+106, \href
  {https://doi.org/10.1090/memo/1232} {\path{doi:10.1090/memo/1232}}.

\bibitem{KW95}
E.~Kirchberg and S.~Wassermann, \emph{Operations on continuous bundles of
  {$C^*$}-algebras}, Math. Ann. \textbf{303} (1995), no.~4, 677--697, \href
  {https://doi.org/10.1007/BF01461011} {\path{doi:10.1007/BF01461011}}.

\bibitem{Leh99}
F.~Lehner, \emph{Computing norms of free operators with matrix coefficients},
  Amer. J. Math. \textbf{121} (1999), no.~3, 453--486, available from
  \url{http://muse.jhu.edu/journals/american_journal_of_mathematics/v121/121.3lehner.pdf}.

\bibitem{LP00}
C.-K. Li and Y.-T. Poon, \emph{Convexity of the joint numerical range}, SIAM J.
  Matrix Anal. Appl. \textbf{21} (2000), no.~2, 668--678, \href
  {https://doi.org/10.1137/S0895479898343516}
  {\path{doi:10.1137/S0895479898343516}}.

\bibitem{LP11}
C.-K. Li and Y.-T. Poon, \emph{Generalized numerical ranges and quantum error
  correction}, J. Operator Theory \textbf{66} (2011), no.~2, 335--351.

\bibitem{LP19}
C.-K. Li and Y.-T. Poon, \emph{Numerical range, dilation, and completely
  positive maps}, Proc. Amer. Math. Soc. \textbf{147} (2019), no.~11,
  4805--4811, \href {https://doi.org/10.1090/proc/14582}
  {\path{doi:10.1090/proc/14582}}.

\bibitem{Passer}
B.~Passer, \emph{Shape, scale, and minimality of matrix ranges}, Trans. Amer.
  Math. Soc. \textbf{372} (2019), no.~2, 1451--1484, \href
  {https://doi.org/10.1090/tran/7665} {\path{doi:10.1090/tran/7665}}.

\bibitem{PS19}
B.~Passer and O.~M. Shalit, \emph{Compressions of compact tuples}, Linear
  Algebra Appl. \textbf{564} (2019), 264--283, \href
  {https://doi.org/10.1016/j.laa.2018.12.002}
  {\path{doi:10.1016/j.laa.2018.12.002}}.

\bibitem{PSS18}
B.~Passer, O.~M. Shalit, and B.~Solel, \emph{Minimal and maximal matrix convex
  sets}, J. Funct. Anal. \textbf{274} (2018), no.~11, 3197--3253, \href
  {https://doi.org/10.1016/j.jfa.2017.11.011}
  {\path{doi:10.1016/j.jfa.2017.11.011}}.

\bibitem{PauBook}
V.~Paulsen, \emph{Completely bounded maps and operator algebras}, Cambridge
  Studies in Advanced Mathematics, vol.~78, Cambridge University Press,
  Cambridge, 2002, \href {https://doi.org/10.1017/CBO9780511546631}
  {\path{doi:10.1017/CBO9780511546631}}.

\bibitem{Pis96}
G.~Pisier, \emph{A simple proof of a theorem of {K}irchberg and related results
  on {$C^*$}-norms}, J. Operator Theory \textbf{35} (1996), no.~2, 317--335.

\bibitem{Pis19}
G.~Pisier, \emph{On a linearization trick}, Enseign. Math. \textbf{64} (2018),
  no.~3-4, 315--326, \href {https://doi.org/10.4171/LEM/64-3/4-5}
  {\path{doi:10.4171/LEM/64-3/4-5}}.

\bibitem{Pop91}
G.~Popescu, \emph{von {N}eumann inequality for {$(B({\mathcal H})^n)_1$}},
  Math. Scand. \textbf{68} (1991), no.~2, 292--304, \href
  {https://doi.org/10.7146/math.scand.a-12363}
  {\path{doi:10.7146/math.scand.a-12363}}.

\bibitem{Sch05}
H.~Schultz, \emph{Non-commutative polynomials of independent {G}aussian random
  matrices. {T}he real and symplectic cases}, Probab. Theory Related Fields
  \textbf{131} (2005), no.~2, 261--309, \href
  {https://doi.org/10.1007/s00440-004-0366-7}
  {\path{doi:10.1007/s00440-004-0366-7}}.

\bibitem{Voi85}
D.~Voiculescu, \emph{Symmetries of some reduced free product
  {$C^\ast$}-algebras}, Operator algebras and their connections with topology
  and ergodic theory ({B}u\c{s}teni, 1983), Lecture Notes in Math., vol. 1132,
  Springer, Berlin, 1985, pp.~556--588, \href
  {https://doi.org/10.1007/BFb0074909} {\path{doi:10.1007/BFb0074909}}.

\bibitem{VDNBook}
D.~V. Voiculescu, K.~J. Dykema, and A.~Nica, \emph{Free random variables}, CRM
  Monograph Series, vol.~1, American Mathematical Society, Providence, RI,
  1992, A noncommutative probability approach to free products with
  applications to random matrices, operator algebras and harmonic analysis on
  free groups, \href {https://doi.org/10.1090/crmm/001}
  {\path{doi:10.1090/crmm/001}}.

\bibitem{Zalar}
A.~Zalar, \emph{Operator {P}ositivstellens\"{a}tze for noncommutative
  polynomials positive on matrix convex sets}, J. Math. Anal. Appl.
  \textbf{445} (2017), no.~1, 32--80, \href
  {https://doi.org/10.1016/j.jmaa.2016.07.043}
  {\path{doi:10.1016/j.jmaa.2016.07.043}}.

\end{thebibliography}

\linespread{1}
\setlength{\parindent}{0pt}

\end{document}